\documentclass[reqno,a4paper,11pt]{amsart}

\usepackage{amsmath,amssymb,amstext,amsthm}
\usepackage[UKenglish]{babel}
\usepackage{hyperref}
\usepackage{amsfonts}
\usepackage{amssymb}
\usepackage{amsthm}
\usepackage{amsmath}
\usepackage{graphicx}
\usepackage{comment}
\usepackage{url}
\usepackage[all]{xy}
\usepackage{lscape}
\usepackage{mathtools}
\usepackage{color}
\usepackage{tikz}
\usepackage{tikz-cd}
\usepackage{float}
\usetikzlibrary{shapes.geometric}
\usetikzlibrary{shapes}
\usetikzlibrary{arrows}
\usetikzlibrary{decorations.markings}
\usetikzlibrary{tqft}
\usetikzlibrary{positioning}
\usetikzlibrary{patterns}
\usepackage{dsfont}
\usepackage{tikz-cd}
\usepackage{enumerate}
\usepackage{fancyhdr}
\usepackage{enumitem}
\usepackage{hyperref}

\newcommand{\Stab}{\mbox{\rm Stab}}
\newcommand{\Ind}{\mbox{\rm Ind}}
\newcommand{\triv}{\textnormal{triv}}
\newcommand{\Deck}{\textnormal{Deck}}
\newcommand{\comm}{\textnormal{comm}}
\newcommand{\im}{\textnormal{im}}
\newcommand{\id}{\textnormal{id}}
\newcommand{\prim}{\textnormal{prim}}
\newcommand{\Irr}{\textnormal{Irr}}
\newcommand{\Aut}{\textnormal{Aut}}

\newcommand{\Res}{\textnormal{Res}}

\newcommand{\modd}{\textnormal{mod}}
\newcommand{\Span}{\mbox{\rm Span}}

\newcommand{\Ab}{\rm Ab}
\newcommand*\from{\colon}
\newcommand{\Z}{\mathbb{Z}}

\newcommand{\C}{\mathbb{C}}
\newcommand{\Q}{\mathbb{Q}}

\newcommand{\N}{\mathbb{N}}

\theoremstyle{plain}
\theoremstyle{definition}

\newtheorem*{theorem*}{Theorem}

\newtheorem*{lemma*}{Lemma}

\newtheorem*{claim*}{Claim}
\newtheorem{propo}{Proposition}[section]
\newtheorem{dfn}[propo]{Definition}
\newtheorem*{dfn*}{Definition}

\newtheorem{lem}[propo]{Lemma}
\newtheorem{corol}[propo]{Corollary}
\newtheorem{theor}[propo]{Theorem}

\newtheorem{examp}[propo]{Example}

\newtheorem*{rem*}{Remark}

\newtheorem*{ques*}{Question}

\title[Subrepresentations in the homology of finite covers of graphs]{Subrepresentations in the homology \\ of finite covers of graphs }
\author[Xenia Flamm]{Xenia Flamm}
\address{Department of Mathematics, ETH Z\"{u}rich, Switzerland}
\email{xenia.flamm@math.ethz.ch}

\begin{document}

\tikzset{->-/.style={decoration={
  markings,
  mark=at position #1 with {\arrow{>}}},postaction={decorate}}}
  
\tikzset{middlearrow/.style={
        decoration={markings,
            mark= at position 0.5 with {\arrow{#1}} ,
        },
        postaction={decorate}
    }
}

\setcounter{MaxMatrixCols}{30}
\definecolor{light-gray}{gray}{0.7}
\definecolor{dark-gray}{gray}{0.2}


\def\subjclassname{\textup{2020} Mathematics Subject Classification}
\expandafter\let\csname subjclassname@1991\endcsname=\subjclassname

\begin{abstract}
Let $p \from Y \to X$ be a finite, regular cover of finite graphs with associated deck group $G$, and consider the first homology $H_1(Y;\C)$ of the cover as a $G$-representation.
The main contribution of this article is to broaden the correspondence and dictionary between the representation theory of the deck group $G$ on the one hand, and topological properties of homology classes in $H_1(Y;\C)$ on the other hand.
We do so by studying certain subrepresentations in the $G$-representation $H_1(Y;\C)$.

The homology class of a lift of a primitive element in $\pi_1(X)$ spans an induced subrepresentation in $H_1(Y;\C)$, and we show that this property is never sufficient to characterize such homology classes if $G$ is Abelian.
We study $H_1^{\comm}(Y;\C) \leq H_1(Y;\C)$---the subrepresentation spanned by homology classes of lifts of commutators of primitive elements in $\pi_1(X)$.
Concretely, we prove that the span of such a homology class is isomorphic to the quotient of two induced representations.
Furthermore, we construct examples of finite covers with $H_1^{\comm}(Y;\C) \neq \ker(p_*)$.
\end{abstract}

\maketitle
\section{Introduction}
We study subrepresentations in the homology of finite covers of finite graphs, viewed as a representation of the deck group.
The subrepresentations we consider arise from homology classes of lifts of closed curves that have interesting topological properties, for example being primitive.
One goal is to better understand the connection between topological properties of curves, and representation-theoretic properties of their homology classes.
For a finite, regular cover $p \from Y \to X$ of finite graphs, with associated deck group $G$, we know that $H_1(Y ; \C$) is spanned by a finite set of closed loops.
Let $n$ be the rank of the free group $\pi_1(X)$.
For a subset $S \subseteq F_n = \pi_1(X)$, we consider the subrepresentation
\[
H_1^S(Y;\C) \coloneqq \Span_{\C [G]} \left\{ [ \tilde{s} ] \mid s \in S \right\} \leq H_1(Y;\C),
\]
where $\tilde{s}$, called the \emph{preferred elevation of $s$}, is the lift of the smallest power of $s$ such that $\tilde{s}$ is a closed loop in $Y$, and $\left[ \tilde{s} \right]$ its homology class.
If $S$ is the set of primitive elements in $F_n$, i.e.\ loops in $X$ that are part of a basis of $\pi_1(X)$, we write $H_1^{\prim}(Y;\C)$ for $H_1^S(Y;\C)$, and call this subrepresentation the \textit{primitive homology} of $Y$, following Farb and Hensel \cite{FarbHensel_FiniteCoversOfGraphs}.
In this paper, the authors asked whether $H_1^{\prim}(Y;\C)=H_1(Y;\C)$.
This question was presumably first asked by March\'e for the homology of surfaces with $\Z$-coefficients, see \cite{Marche_HomologyGeneratedByLiftsOfSimpleCurves}, and later by Looijenga in \cite{Looijenga_SomeAlgGeomRelatedToMCG}.
Partial results were obtained by Farb and Hensel in \cite{FarbHensel_FiniteCoversOfGraphs}, where they showed that in fact equality holds for finite Abelian and $2$-step nilpotent groups, the latter only when $n \geq 3$; see \cite[Propositions 3.2 and 3.3]{FarbHensel_FiniteCoversOfGraphs}.
Malestein and Putman gave a complete answer to this question in \cite[Theorem C, Example 1.3]{MalesteinPutman_SCCFiniteCoversOfSurfaces}.
For every $n \geq 2$ they constructed a finite group which answers the question in the negative.

A central role is played by the theorem of Gasch\"utz, see \cite{Chevalley_VerhaltenIntegrale} or \cite[Theorem 2.1]{GrunewaldLarsenLubotzkyMalestein_ArithmeticQuotientsOfMCG}, which identifies the first homology of $Y$ as a representation of $G$; namely, we have an isomorphism of $G$-representations 
\[ H_1(Y;\C) \cong \C_{\triv} \oplus \C[G]^{\oplus (n-1)}, \]
where we denote by $\C_{\triv}$ the trivial one-dimensional representation of $G$ and by $ \C[G]$ the regular representation of $G$.

To obstruct primitive homology from being all of homology, Farb and Hensel used the theorem of Gasch\"utz together with the following representa-tion-theoretic property:
If $l$ is a primitive loop in $X$, then the $G$-orbit of the homology class of its preferred elevation is linearly independent, in other words 
\[\Span_{\C [G]} \{[\tilde{l]}\} \cong \Ind^G_{\langle \phi(l) \rangle} (\C_{\triv}), \]
where $\phi \from \pi_1(X) \to G$ is the epimorphism associated to the cover $p \from Y \to X$, compare \cite[Proposition 2.1]{FarbHensel_FiniteCoversOfGraphs}.
From this we obtain a necessary representation-theoretic condition for a homology class to be the homology class of a lift of a primitive element.
Recall that a homology class in $H_1(X;\Z)$ can be represented by a primitive loop in $X$ if and only if it is indivisible.
The question arises whether there is a similar characterization for homology classes of lifts of primitive elements.
Our first result shows that the necessary representation-theoretic property of Farb and Hensel is never sufficient for regular, Abelian covers.

\begin{theor} \label{Propo2.1_Inverse}
Let $G$ be the deck group of a finite, regular cover $Y \to X$ of finite graphs with rank of $\pi_1(X)$ at least two.
If $G$ is non-trivial Abelian, then there exists $z \in H_1(Y;\mathbb{Z})$ indivisible such that
\begin{enumerate}[label=(\roman*)]
\item \label{item1_Propo2.1_Inverse} $\Span_{\C [G]} \{z\} \cong \Ind^G_{\langle g \rangle} (\C_{\triv})$ for some $g \in G$, and
\item \label{item2_Propo2.1_Inverse} $z$ is not represented by an elevation of a primitive loop in $X$.
\end{enumerate}
\end{theor}

We ask whether Theorem \ref{Propo2.1_Inverse} also holds for non-Abelian groups.


We extend some of the above questions to other subsets $S \subseteq F_n$.
We consider the subset of \emph{primitive commutators}, by which we mean a commutator of the form $[w, w']$ for $w, w'$ elements in a common basis of $F_n$.
Homology classes of elevations of primitive commutators satisfy the following representation-theoretic property, in analogy to \cite[Proposition 2.1]{FarbHensel_FiniteCoversOfGraphs} for primitive elements.

\begin{propo} \label{Propo_CommutatorHom}
Let $G$ be the deck group of a finite, regular cover $Y \to X$ of finite graphs with associated epimorphism $\phi \from \pi_1(X) \to G$.
Let $x_1$, $x_2 \in F_n$ be two primitive elements that extend to a free basis of $F_n$.
Set $x\coloneqq [x_1,x_2]$, and let $K \coloneqq \langle \phi(x) \rangle \leq \langle \phi(x_1), \phi(x_2) \rangle \eqqcolon H \leq G$ be the subgroups of $G$ generated by $\phi(x)$, respectively $\phi(x_1)$ and $\phi(x_2)$.
Then we have the following isomorphism of $G$-representations
\[ \Span_{\C [G]} \{ [\tilde{x}] \} \cong \Ind_K^G(\C_{\triv}) / \Ind_H^G(\C_{\triv}).\]
\end{propo}

Note that if $K \leq H\leq G$ are subgroups then $\Ind_H^G(\C_\triv)$ is a $G$-sub-representation of $\Ind_K^G(\C_\triv)$.
Indeed, by transitivity of induction it suffices to realize that $\C_\triv$ is an $H$-subrepresentation of $\Ind_K^H(\C_\triv)$ using for example Frobenius reciprocity,  compare \cite[Chapters 7.1-7.2]{Serre_LinearRepresentationsOfFiniteGroups}.

Our next result shows that this representation-theoretic property is not sufficient in the case $n=2$.

\begin{propo} \label{Propo_InversePropo1.5}
For $n=2$ there exists a finite, regular cover $p \from Y \to X$ and $z \in H_1(Y;\Z) \cap \ker(p_*)$ indivisible,  where $p_*$ is the induced map on the associated homology groups, such that
\begin{enumerate}[label=(\roman*)]
\item $\Span_{\C [G]} \{z\} \cong \Ind_{\langle g \rangle}^G(\C_\triv)/\C_\triv $ for some $g \in G$, where $G$ is the deck group associated to the cover $p$, and
\item $z$ cannot be represented by an elevation of a primitive commutator in $X$.
\end{enumerate}
\end{propo}


We turn to the space $H_1^S(Y;\C)$ for $S$ the subset of all primitive commutators over all bases of $F_n$ (not just some fixed one), and write $H_1^\comm(Y;\C) \leq H_1(Y;\C)$ for this subrepresentation, the \emph{primitive commutator homology} of $Y$.
It is clear that $H_1^\comm(Y;\C) \leq \ker(p_*) \neq H_1(Y;\C)$.
One can ask whether $H_1^\comm(Y;\C) = \ker(p_*)$ holds.

To obstruct primitive commutator homology to constitute all of $\ker(p_*)$, we introduce the following notation following \cite[Definition 1.3]{FarbHensel_FiniteCoversOfGraphs}.
Let $\Irr(G)$ be the set of representatives of the pairwise non-isomorphic irreducible $G$-representations, and for $S \subseteq F_n$, let 
\[\Irr^S(\phi,G) \subseteq \Irr(G) \]
be the subset of those irreducible representations $V$ of $G$ that have the property that there is an element in $S$ whose image has a non-zero fixed point.
More precisely, $V \in \Irr^S(\phi,G)$ if and only if there exists an element $s \in S$ and $0 \neq v \in V$ such that $\phi(s)(v)=v$.
If $S$ is the set of primitive elements or primitive commutators, we write $\Irr^\prim(\phi,G)$ or $\Irr^\comm(\phi,G)$, respectively.
For a $G$-representation $M$ and $V \in \Irr(G)$,  we write $M(V)$ for the sum of all subrepresentations of $M$ isomorphic to $V$, and call them the homogeneous components of $M$.

\begin{theor} \label{Theorem_RepresentationTheoreticObstructionCommutatorHom}
Let $G$ be the deck group of a finite, regular cover $p \from Y \to X$ of finite graphs with associated epimorphism $\phi \from \pi_1(X) \to G$.
Then
\[ H_1^\comm(Y;\C) \leq \bigoplus_{V \in \Irr^{\comm}(\phi,G) \setminus \{\C_\triv\}} M(V), \]
for $M=H_1(Y;\C)$.
\end{theor}

The argument follows closely the one of \cite[Theorem 1.4]{FarbHensel_FiniteCoversOfGraphs}.
The main ingredient of the proof of this theorem is Proposition \ref{Propo_CommutatorHom}.
The theorem implies that whenever a group $G$ together with an epimorphism $\phi \from F_n \to G$ satisfies $\Irr^{\comm}(\phi,G) \neq \Irr(G)$, we conclude by the theorem of Gasch\"utz that $H_1^{\comm}(Y;\C) \neq \ker(p_*)$.
We will give an example of a group with this property in the case $n=2$.
In fact, even more is true.

\begin{theor} \label{Thm_H1CommNeqKer}
For every $n \geq 2$ there exists a finite, regular cover $p \from Z \to X$ such that $H_1^{\comm}(Z;\C) \neq \ker(p_*)$.
\end{theor}

The proof of this theorem relies on the result of Malestein-Putman on primitive homology, see \cite[Theorem C, Example 1.3]{MalesteinPutman_SCCFiniteCoversOfSurfaces}.

\subsection*{Organization}
The paper is organized as follows.
Section \ref{Section_primitive} is solely concerned with homology classes of elevations of primitive elements.
We will introduce the important results, and finish by giving a proof of Theorem \ref{Propo2.1_Inverse}.
In Section \ref{Section_HomClassesElevPrimComm}, we will prove a necessary representation-theoretic property of homology classes of elevations of primitive commutators (Proposition \ref{Propo_CommutatorHom}), but which is not sufficient to characterize those homology classes, see Proposition \ref{Propo_InversePropo1.5}.
In Section \ref{Section_CommutatorHomology}, we study the primitive commutator homology, and we will prove Theorem \ref{Theorem_RepresentationTheoreticObstructionCommutatorHom}.
We will discuss how primitive homology and primitive commutator homology relate, which will be essential in the proof of Theorem \ref{Thm_H1CommNeqKer}.

\subsection*{Acknowledgements}
I would like to thank Sebastian Hensel for introducing me to this topic and for his helpful remarks.
Thanks goes as well to Bram Petri for valuable input, and Peter Feller, Gerhard Hiss and Anna Ribelles P\'erez for constructive discussions and feedback.
I am grateful to the referee for the detailed comments and suggestions, particularly for pointing out an error in the original argument for Corollary \ref{Corollary2.1_Inverse}.


\section{Homology Classes of Elevations of Primitive Elements} \label{Section_primitive}

Let $n \in \N$, $n \geq 2$, a finite group $G$ and a surjective homomorphism $\phi \from F_n \to G$, where $F_n = F \langle x_1, \ldots, x_n \rangle$ is the free group on $n$ generators.
Let $X$ be the wedge of $n$ copies of $S^1$ with vertex $x_0$.
Then we can associate to $\phi$ a finite regular path-connected cover $p \from Y \to X$ with base point $y_0 \in p^{-1}(x_0)$ and $p_*(\pi_1(Y,y_0))=\ker(\phi)$.
Under these assumptions, we have $G \cong F_n / \ker(\phi) \cong \pi_1(X,x_0)/p_*(\pi_1(Y,y_0))$ and $G$ acts on $Y$ by graph automorphisms.
This action extends to a linear action of $G$ on the finite-dimensional $\C$-vector space $H_1(Y;\C)$.

Recall that we call an element of a free (Abelian) group \emph{primitive} if it is part of a free basis.
An element of a free Abelian group is primitive if and only if it is \emph{indivisible}, that means it cannot be written as a non-trivial multiple of some other element.
Being indivisible is equivalent to the coefficient vector with respect to a basis having greatest common divisor one.
A primitive element in $F_n$ has a primitive image under the quotient map to $\Z^n$.
On the other hand, every primitive element in $\Z^n$ has a primitive preimage in $F_n$.

\begin{dfn}
For an element $l \in F_n$, let $k(l)$ be the minimal number such that $l^{k(l)} \in \ker(\phi)$.
The \textit{preferred elevation $\tilde{l}$ of $l$} is the lift of $l^{k(l)}$ to $Y$ at the base point $y_0$.
Lifts of $l^{k(l)}$ at other preimages of $x_0$ are just called \emph{elevations}.
Unless clear from context we write $\tilde{l}_p$, where $p \from Y \to X$ is the covering map, to specify which cover we are lifting to.
\end{dfn}

Note that $k(l)$ is finite for all $l \in F_n$, since $G$ is finite.
By the regularity of the cover, we obtain all elevations of $l$ by applying the elements of $G$ to the preferred elevation $\tilde{l}$.
We can view $\tilde{l}$ as an element in $\pi_1(Y,y_0)$, where it will also be called $\tilde{l}$.
Its homology class is denoted by square brackets $[\tilde{l}]$.
It is straightforward to verify that elevations of primitive elements are primitive.

For a cover $p \from Y \to X$ we denote by $p_*$ the induced morphisms on the level of fundamental groups and homology groups.
Since $Y$ is finite we have the \emph{transfer map} $p_{\#} \from C_\bullet(X) \to C_\bullet(Y), \; x \mapsto \sum_{\tilde{x} \textrm{ lift of } x} \tilde{x}.$
The map $p_{\#}$ commutes with the differential operators and thus induces a map on homology, which will also be denoted by $p_{\#}$.
We have $p_* \circ p_{\#} = |G| \cdot \id$, which in particular implies that $p_* \from H_1(Y;\C)\to H_1(X;\C)$ is surjective and $p_{\#} \from H_1(X;\C)\to H_1(Y;\C)$ is injective. 

For $l\in F_n$ we observe that $p_*([\tilde{l}]) = |  \phi (l)  | \cdot [ l ] \in H_1(X;\Z) $, since $k(l) = | \phi(l) |$.
If $G$ is Abelian, and $\tilde{\phi} \from \Z^n \to G$ denotes the homomorphism obtained from $\phi$ which factors through $\Z^n$ then
\begin{equation} \label{eqn_AbelianFactorThroughHomology}
\phi(l) = \tilde{\phi}\left([l]\right) = \tilde{\phi} \left( \frac{p_*([\tilde{l}])}{| \phi (l)  |} \right).
\end{equation}

We summarize the general setup of this article in the following diagram, which can be consulted at every point.
\[
\begin{tikzcd}
Y \arrow[two heads]{d}{p} & F_m=\pi_1(Y) \arrow[two heads]{r}{\Ab} \arrow[hook]{d}{p_*} & \Z^m=H_1(Y;\Z) \arrow{r}{-\otimes_\Z \C} \arrow{d}{p_*} & H_1(Y;\C)=\C^m \arrow[two heads]{d}{p_*}\\
X  & F_n=\pi_1(X) \arrow[two heads]{r}{\Ab} \arrow[two heads]{d}{\phi} & \Z^n=H_1(X;\Z) \arrow{r}{-\otimes_\Z \C} \arrow[shift left=2]{u}{p_{\#}} \arrow[two heads]{d}{\tilde{\phi}} & H_1(X;\C)=\C^n \arrow[shift left=2, hook]{u}{p_{\#}} \\
 & G \arrow[two heads]{r}{\Ab} & G^{\Ab}
\end{tikzcd}
\]
In the rest of this section we will show that the property given in \cite[Proposition 2.1]{FarbHensel_FiniteCoversOfGraphs} is not sufficient to characterize elevations of primitive elements.
We recall the statement here.

\begin{propo}[Farb-Hensel] \label{Propo2.1}
Let $G$ be the deck group of a finite, regular cover $p \from Y \to X$ of finite graphs with associated epimorphism $\phi \from \pi_1(X) \to G$.
Let $l$ be a primitive loop in $X$ and let $\tilde{l}$ be its preferred elevation in $Y$.
We set $z \coloneqq [\tilde{l}]$ and $g \coloneqq \phi(l)$.
Then there is an isomorphism of $G$-representations
\[ \Span_{\C [G]} \left\{ z \right\} \cong \Ind_{\langle g \rangle}^G(\C_\triv). \]
\end{propo}

Note that this proposition contains two statements.
Firstly, it says that the $G$-orbit of an elevation of a primitive element is $\C$-linearly independent.
Secondly, it tells us that the subgroup over which we induce is related to the primitive element $l$ we started with.
It is cyclic and generated by $\phi(l)$.
If $G$ is Abelian, this allows us to relate $z$ and $\phi(l)$,  see Equation (\ref{eqn_AbelianFactorThroughHomology}), --- a fact which will be exploited in the proof of Theorem \ref{Propo2.1_Inverse} to verify the insufficiency of this property.

Let $V_1=\C_\triv, \ldots, V_{k(G)}$ be representatives of the isomorphism classes of the irreducible $G$-representations, where $k(G)$ is the number of non-isomorphic irreducible $G$-representations.
If $M$ is a $G$-representation, we write $M(V_i)$ for the sum of all subrepresentations of $M$ isomorphic to $V_i$, and call them the homogeneous components of $M$.
Recall that if $M=\C[G]$, then $M(V_i) \cong V_i^{\oplus \dim(V_i)}$ for all $ 1 \leq i \leq k(G)$, compare e.g.\ \cite[Chapter 2]{Isaacs_CharacterTheoryOfFiniteGroups}.
The following lemma follows directly from the theorem of Gasch\"utz.

\begin{lem} \label{Lemma_TechnicalCoverings1}
Let $M=H_1(Y;\C)$.
In the above setting, the following properties hold:
\begin{enumerate}[label=(\alph*)]
\item \label{item2_technicalCoverings1} $p_{\#}(H_1(X;\C)) \cong M(V_1)$.
\item \label{item4_technicalCoverings1} $\ker(p_*) \cong \bigoplus_{i=2}^{k(G)} M(V_i)$.
\end{enumerate}
Together we obtain $H_1(Y;\C) = p_{\#}(H_1(X;\C)) \oplus \ker(p_*)$.
\end{lem}
\begin{proof}
\begin{enumerate}[label=(\alph*), leftmargin=*]
\item The group $G$ acts trivially on $p_\#(H_1(X;\C))$, since an element of the deck group permutes lifts.
Hence $p_{\#}(H_1(X;\C)) \leq M(V_1)$.
Since $p_\#$ is injective as a map on homology we obtain equality by the theorem of Gasch\"utz and dimensionality reasons.

\item This follows from the theorem of Gasch\"utz and the fact that the map $p_* \from H_1(Y;\C) \to H_1(X;\C)$ is surjective.
Indeed, we have
\[M(V_1) \cong \C^n \cong H_1(X;\C) = \im(p_*) \cong M/ \ker(p_*),\]
as $\C [G]$-modules, where $G$ acts trivially on $H_1(X;\C)$, so $\ker(p_*)$ does not have simple submodules isomorphic to $V_1$.
For dimensionality reasons we obtain $\ker(p_*) \cong \bigoplus_{i=2}^{k(G)} M(V_i)$.
\end{enumerate}
\noindent Let $x\in H_1(X;\C)$ with $p_\#(x) \in \ker(p_*)$.
Then $p_*(p_\#(x))= |G|x =0$, and hence $x=0$.
Thus $\ker(p_*) \cap \im( p_{\#}) = \{0\}$.
\end{proof}

Before we prove our first result we state the following lemma about induced representations.

\begin{lem} \label{Lemma_SpanInducedRepr}
Let $G$ be a finite group and $M$ a representation of $G$ over $\C$.
Assume there exists $0 \neq m \in M$ such that $Gm$ is linearly independent. 
Let $H \coloneqq \Stab_G(m)$. 
Then
\[\Span_{\C[G]}\{m\} \cong \Ind^G_H(\C_\triv).\]
\end{lem}
\begin{proof}
Choose representatives $g_1 = 1, g_2, \ldots, g_{[G:H]}$ for the left cosets of $H$ in $G$. 
For all $g \in G$, $g $ can be uniquely written as $g = g_i h$ for some
$1 \leq i \leq [G : H]$ and $h \in H$. 
We then have $gm = g_i h m = g_i m$, since $h \in \Stab_G(m)$. 
Also note that if $g_i m = g_j m$, then $g_j^{-1}g_i \in H$, so $g_i \in g_jH$ which implies $i = j$. 
Thus ${g_1 m, \ldots, g_{[G:H]}m}$ is a basis for $\Span_{\C[G]}\{m\}$.
By construction, $\Ind^G_H(\C_\triv) = U_1 \oplus \ldots \oplus U_{[G:H]} $ as $\C[H]$-modules with $U_1 = \C_\triv$ as $\C[H]$-module,  and $U_i = g_i U_1$. 
Let $u_1 \in U_1 \setminus \{0\}$. 
Then $U_1 = \langle u_1\rangle$ and thus $U_i = \langle g_iu_1 \rangle$ for all $1 \leq i \leq [G : H]$. 
Define $\varphi \from \Span_{\C[G]}\{m\} \to \Ind^G_H(\C_\triv)$, $g_i m \mapsto g_i u_1$, and $\C$-linear extension.
This is an isomorphism of $\C[G]$-modules.
\end{proof}

We are now ready to prove our first result.
\begin{proof}[Proof of Theorem~\ref{Propo2.1_Inverse}]
Recall that we would like to prove that if $G$ is the deck group of a finite, regular cover $Y \to X$ of finite graphs with rank of $\pi_1(X)$ at least two, and $G$ is non-trivial Abelian, then there exists an indivisible homology class $z \in H_1(Y;\mathbb{Z})$ such that
\begin{enumerate}[label=(\roman*)]
\item $\Span_{\C [G]} \{z\} \cong \Ind^G_{\langle g \rangle} (\C_{\triv})$ for some $g \in G$, and
\item $z$ is not represented by an elevation of a primitive loop in $X$.
\end{enumerate}
The proof has two parts: First, we construct an indivisible homology class $z \in H_1(Y;\Z)$ starting from two homology classes of elevations of primitive elements.
Second, we verify that $z$ indeed satisfies conditions \ref{item1_Propo2.1_Inverse} and \ref{item2_Propo2.1_Inverse}.

\begin{claim*}
\textnormal{
There are primitive loops $l', l''$ in $X$ with $\langle \phi(l') \rangle \cap \langle \phi(l'') \rangle = \{1\}$ and $\langle \phi(l') \rangle \neq \{1\}$. }
\end{claim*}

We will now prove the claim.
Since $G$ is Abelian, the map $\phi$ factors through $\Z^n$, so we obtain a surjective homomorphism $\tilde{\phi} \from \Z^n \to G$ of Abelian groups.
By the invariant factor decomposition of finitely generated modules (refer e.g.\ to \cite[Chapter 2.9, Theorem 2]{Bosch_Algebra}) the statement is clear for $\tilde{\phi}$ because $G$ is non-trivial and $n\geq 2$.
Since the Abelianization $F_n \to \Z^n$ has the property that every primitive element in $\Z^n$ has a primitive preimage, the claim is proven.

Choose $l', l''$ primitive loops in $X$ that satisfy the above claim.
To simplify notation we set $g' \coloneqq \phi(l')$ and $g'' \coloneqq \phi(l'') \in G$.
We consider the homology classes $z'=[\tilde{l'}]$ and $z''=[\tilde{l''}] \in H_1(Y;\Z)$ of the preferred elevations $\tilde{l'}$ and $\tilde{l''}$ of the two primitive loops $l'$ and $l''$, respectively.
The elements $z'$ and $z''$ will be used to build our desired homology class $z$.
Write $z' = z_1' + z_2'+ \ldots + z_{|G|}'$ and $z'' = z_1'' + z_2''+ \ldots + z_{|G|}''$
for the decomposition into the homogeneous components, i.e.\ $z_i'$, $z_i'' \in M(V_i)$ for $i \leq 1 \leq |G|$.
By the theorem of Gasch\"utz, we have $z_1', z_1'' \in V_1^{\oplus n}$ and $z_i', z_i'' \in V_i^{\oplus (n-1)}$ for $2 \leq i \leq |G|$.

We claim that $z_1' \in H_1(Y;\Q)$.
Indeed, by Lemma \ref{Lemma_TechnicalCoverings1} \ref{item2_technicalCoverings1}, we know that there exists $x \in H_1(X;\C)$ such that $z_1' = p_\#(x)$, and $|G| \, x = p_* \circ p_\# (x) = p_* (z_1') = p_*(z') \in H_1(X;\Z)$, using that $z_2', \ldots, z_{|G|}' \in \ker(p_*)$, see \ref{item4_technicalCoverings1} in the same lemma.
Thus, $ |G| \, z_1' = p_\#( |G| \, x) \in H_1(Y;\Z).$
By the same argument $z_1'' \in H_1(Y;\Q)$.

Define $b \coloneqq z_1'' - z_1' \in V_1^{\oplus n}$ and set 
\[ \hat{z} \coloneqq z' + b \in H_1(Y; \Q). \]
Let $q \in \Q$ such that $z \coloneqq q \hat{z} \in H_1(Y;\Z)$ is indivisible.
This will be our candidate homology class.

We come to the second part of the proof.
First, we verify that \ref{item1_Propo2.1_Inverse} holds for $z$.
For this we need to show that the projection of $\hat{z}$ to $M(V_1)$, namely $z_1''$, is non-trivial.
Assuming this, it is a straight-forward computation that $\Stab_G(\hat{z}) = \Stab_G(z') = \langle g' \rangle$, and that the $G$-orbit of $\hat{z}$ is linearly independent over $\C$, since this is true for $z'$.
Now \ref{item1_Propo2.1_Inverse} follows from Lemma \ref{Lemma_SpanInducedRepr} with
\begin{equation} \label{z-Span}
\Span_{\C [G]} \{z\} = \Span_{\C [G]} \{\hat{z}\} \cong \Ind_{\langle g' \rangle}^G (\C_{\triv}).
\end{equation}
To show that $z_1'' \neq 0$ we will argue by contradiction.
In fact, if $z_1'' = 0$, Lemma \ref{Lemma_TechnicalCoverings1} \ref{item4_technicalCoverings1} implies that $0=p_*(z_1'')=p_*(z'')= |g''  | [l'']$, where the last equality holds since $z''$ is the homology class of the preferred elevation of $l''$.
Thus $[l'']=0$ which contradicts the fact that $l''$ is primitive.

To verify \ref{item2_Propo2.1_Inverse}, we clearly need that $b \neq 0$ because otherwise $z=z'$.
Assume for a contradiction that $b=0$, or equivalently $z_1'=z_1''$.
By Lemma \ref{Lemma_TechnicalCoverings1} \ref{item4_technicalCoverings1}, it follows that $p_*(z')=p_*(z'')$, and thus $|g'| [l'] =| g'' | [l''] \in H_1(X;\Z)$.
Since $[l']$ and $[l'']$ are indivisible integral homology classes,  they are equal.
As $G$ is Abelian,  we have $g' = \tilde{\phi}([l'])= \tilde{\phi}([l''])=g''$ by Equation (\ref{eqn_AbelianFactorThroughHomology}), which contradicts the choice of $l'$ and $l''$.

Knowing that $b\neq 0$, assume now towards a contradiction that $z$ is the homology class of an elevation of a primitive element $l$ in $X$.
By Proposition \ref{Propo2.1} we have $ \Span_{\C [G]}  \{z\} \cong \Ind_{\langle \phi(l) \rangle}^G ( \C_{\triv} ). $
On the other hand, by Equation (\ref{z-Span}), we have 
\[ \Ind_{\langle g \rangle}^G ( \C_{\triv} ) \cong \Ind_{\langle g' \rangle}^G ( \C_{\triv} ), \]
with $g \coloneqq \phi(l)$.
Now $G$ Abelian implies that $\langle g \rangle = \langle g' \rangle$, since representations of finite groups over $\C$ are determined by their characters and there is an explicit formula for the character of an induced representation, compare e.g.\ \cite[Chapter 5]{Isaacs_CharacterTheoryOfFiniteGroups}.
Further, we compute 
\[ | g | [l]= p_*(z) = q \, p_*(\hat{z}) = q \, p_*(z_1'') = q \, p_*(z'') = q |g'' |  [l''] \in H_1(X;\Z).\]
Again by the indivisibility of the integral homology classes $[l]$ and $[l'']$, we conclude that they are equal.
Since $G$ is Abelian, we obtain $g = g''$ again by Equation (\ref{eqn_AbelianFactorThroughHomology}).
This contradicts the fact that $\langle g' \rangle \cap \langle g'' \rangle = \{1\}$ and $\langle g' \rangle \neq \{1\}$.
\end{proof}

In fact, a slightly stronger result can be verified using the same techniques as in the proof of the above proposition.

\begin{corol} \label{Corollary2.1_Inverse}
Let $z \in H_1(Y;\Z)$ be the element constructed in the proof of Theorem \ref{Propo2.1_Inverse}.
Then there is no element in $\Span_{\C [G]} \{z\}$ which is the homology class of an elevation of a primitive loop in $X$.
\end{corol}
\begin{proof}
Assume there exists $z^* \in \Span_{\C [G]} \{z\} \cap H_1(Y;\Z)$ indivisible with $[ \tilde{l^*} ] = z^*$ for $l^* \in F_n$ primitive.
Then $\Span_{\C [G]} \{z^*\} \cong \Ind_{\langle g^* \rangle}^G(\C_\triv)$ with $g^*\coloneqq \phi(l^*) \in G$ by Proposition \ref{Propo2.1}.
We also have there exist $\alpha_1, \ldots, \alpha_{|G|} \in \C$ such that $z^* = \sum_{i=1}^{|G|} \alpha_i g_i z$ for $g_i \in G$.
Thus 
\[p_*(z^*) = \sum_{i=1}^{|G|} \alpha_i p_*(g_i z) = \underbrace{ \left( \sum_{i=1}^{|G|} \alpha_i \right) }_{\coloneqq A} p_*(z), \]
since $G$ acts as homeomorphisms of the covering space which preserve the projection map $p$.
Then $A \in \Q$ since both $p_*(z^*)$,  $p_*(z)$ are in $H_1(X;\Z)$.
Putting everything together we obtain
\[ |g^*| [l^*] = p_*(z^*) = A q \, p_*(\hat{z}) = A q \, p_*(z'') = A q |g''| [ l''].\]
Since both $[l^*]$ and $[l'']$ are integral indivisible homology classes, they are equal, and thus $g^* = g''$ again by Equation (\ref{eqn_AbelianFactorThroughHomology}).
Recall that $\Stab_G(z)=\langle g' \rangle$.
We claim that $G$ Abelian implies that $\Stab_G(z) \leq \Stab_G(z^*)$ and $\Stab_G(z^*) = \langle g^* \rangle$.
Indeed for $g \in \Stab_G(z)$ we have
\[g z^* = \sum_{i=1}^{|G|} \alpha_i g g_i z = \sum_{i=1}^{|G|} \alpha_i g_i g z = \sum_{i=1}^{|G|} \alpha_i g_i z = z^*,\]
which proves the first part of the claim.
For the second part of the claim we observe that if $v_1,\ldots,v_{[G:\langle g^* \rangle]}$ is a $\C$-basis for $\Span_{\C [G]} \{z^*\} \cong \Ind_{\langle g^* \rangle}^G(\C_\triv)$ which gets permuted transitively by $G$ and for which $\Stab_G(v_1)=\langle g^* \rangle$, then $\Stab_G(v_k) = \langle g^* \rangle$ for all $k=2,\ldots, [G:\langle g^* \rangle]$, and hence $\Stab_G(z^*) \leq \langle g^* \rangle$.
By a dimension argument we obtain equality.
Thus we obtain that $\langle g' \rangle = \Stab_G(z) \leq \Stab_G(z^*) = \langle g^* \rangle = \langle g'' \rangle$,  which contradicts the choice of $l'$ and $l''$ in the proof of Theorem \ref{Propo2.1_Inverse}.
\end{proof}


\section{Homology Classes of Elevations of Primitive Commutators}
\label{Section_HomClassesElevPrimComm}

We begin this section by giving a proof of Proposition \ref{Propo_CommutatorHom}, which gives a necessary representation-theoretic property of homology classes of elevations of primitive commutators.

\begin{proof}[Proof of Proposition \ref{Propo_CommutatorHom}]
Recall that we would like to prove that if $x_1$ and $x_2$ extend to a free basis of $F_n$ and $x_{12} \coloneqq [x_1, x_2]$, then
\[ \Span_{\C [G]} \{ [\tilde{x}_{12}] \} \cong \Ind_K^G(\C_{\triv}) / \Ind_H^G(\C_{\triv}),\]
where $K \coloneqq \langle \phi(x_{12}) \rangle \leq \langle \phi(x_1), \phi(x_2) \rangle \eqqcolon H \leq G$.

Choose a base point preserving homotopy equivalence from $X$ to the wedge of $n$ circles that maps $x_1$ and $x_2$ to two of the circles (e.g.\ by identifying a maximal spanning tree to a point and then applying a base point preserving homotopy equivalence of the wedge of circles).
Then there exists a unique (up to homeomorphism) finite regular cover of the wedge of $n$ circles preserving base points such that the homotopy equivalence lifts to a homotopy equivalence of the corresponding covers and preserves base points.
Since homotopy equivalences induce isomorphisms on the level of fundamental and homology groups,  we can assume without loss of generality that $X$ is a wedge of $n$ circles and $x_1$, $x_2$ are two of those circles.

The proof now consists of two parts.
Firstly we reduce to the case $n=2$, where the statement simplifies to
\[ \Span_{\C [G]} \{ [\tilde{x}_{12}] \} \cong \Ind_K^G(\C_{\triv}) / \C_{\triv}, \]
since $H=G$.
Secondly we use surface topology to prove the claim in this special case.
We start by reducing to the case $n=2$.

Let $X_1$ be the union of the two circles $x_1$ and $x_2$.
Define $Y_1$ as the component of $p^{-1}(X_1)$ that contains the base point $y_0$.
Now choose representatives for the left cosets of $H$ in $G$, say $g_1=1, \ldots, g_{[G:H]}$, and set $Y_i \coloneqq g_i Y_1$ for $1 \leq i \leq [G:H]$.
These are pairwise disjoint subgraphs of $Y$, since the vertices of $Y$ correspond to the group elements and every group element is contained in exactly one $Y_i$.
It is easy to see that
$I \from \bigoplus_{i=1}^{[G:H]} H_1(Y_i;\C) \hookrightarrow H_1(Y;\C)$
is injective, where $I$ is induced by the inclusions of the subgraphs into $Y$; see e.g.\ \cite[Claim 2.4]{FarbHensel_MovingHomologyClasses}.
Note that $I$ is a morphism of $G$-representations, since the action of $G$ on the direct sum is given by permuting the summands according to the permutation of the subgraphs $Y_i$.
Now set $v \coloneqq [ \tilde{x}_{12} ]$.
By definition of $Y_1$, we have $v \in H_1(Y_1;\C)$ and, for every $h \in H$, we have $hv \in H_1(Y_1;\C)$.
For $g \in G$ write $g = g_j h$ for some $1\leq j \leq [G:H]$ and $h \in H$.
Thus $g v = g_j h v \in H_1(g_j Y_1; \C) = H_1(Y_j; \C)$.
We know that
$ \Span_{\C [G] } \{v\} = \sum_{i=1}^{[G:H]} g_i \, \Span_{\C [H]}\{v\}, $
and, since $I$ is injective, we conclude that this sum is in fact direct.
Hence
\[ \Span_{\C [G] } \{v\} \cong \Ind_H^G (\Span_{\C [H]}\{v\}).\]
Proposition \ref{Propo_CommutatorHom} in the case $n=2$ applies to $H$ since it is generated by two elements, and we obtain that $\Span_{\C [H]}\{v\} \cong \Ind_K^H(\C_\triv) / \C_\triv$.
By the exactness of induction (see e.g. \cite[Chapter 7.1]{Serre_LinearRepresentationsOfFiniteGroups}), it then follows that
\[ \Span_{\C [G] }\{v\} \cong \Ind_K^G(\C_\triv) /\Ind_H^G(\C_\triv). \]

Thus, without loss of generality, we can assume that $n=2$ and $G=H$.
To prove the proposition in this case, we use surface topology.
We would like to show that $ \Span_{\C [G]} \{[\tilde{x}_{12}] \} \cong \Ind_K^G(\C_{\triv}) / \C_{\triv} $.
We identify $F_2$ with the fundamental group of the torus with one boundary component $T$ and base point $t_0 \in \partial T$, with generators $x_1$ and $x_2$ represented by the simple closed curves as in the image below.
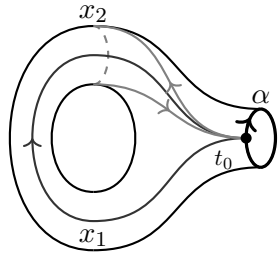
\begin{figure}[H]
\centering
\begin{tikzpicture}[thick,scale=1.1]
\pic[
  transform shape,
  name=a,
  tqft,rotate=90,
  cobordism edge/.style={draw},
  incoming boundary components=0,
  outgoing boundary components=2,
];
\pic[
  transform shape,
  name=b,
  tqft,rotate=90,
  cobordism edge/.style={draw},
  incoming lower boundary component 2/.style={draw,dashed,gray},
  outgoing lower boundary component 1/.style={draw, line width=1.2pt},
  outgoing upper boundary component 1/.style={draw,->-=.8, line width=1.2pt},
  incoming boundary components=2,
  outgoing boundary components=1,
  every outgoing boundary component/.style={draw},
  offset=.5,
  at=(a-outgoing boundary),
  anchor=incoming boundary,
];
\draw[dark-gray,->-=.5] (b-outgoing boundary 1.90) to[out=180,in=0,looseness=1.25] (a-outgoing boundary 1.center) to[out=180,in=180,looseness=1.25] (a-outgoing boundary 2.center) to[out=0,in=180,looseness=1.25] (b-outgoing boundary 1.90);
\draw[gray,->-=.5] (b-outgoing boundary 1.90) to[out=180,in=0,looseness=1.25] (b-incoming boundary 2.0);
\draw[gray,->-=.5] (b-incoming boundary 2.180) to[out=0,in=180,looseness=1.25] (b-outgoing boundary 1.90);
\node (node0) at (b-outgoing boundary 1.90) [circle,fill,scale=0.4] {};
\node [below left=0.001cm of node0,scale=0.8] {$t_0$};
\node at (2,-0.2) {$x_1$};
\node at (2,2.5) {$x_2$};
\node at (4,1.5) {$\alpha$};
\end{tikzpicture}
\end{figure}
Then $x \coloneqq x_{12} = [x_1, x_2]$ is represented by $\alpha=\partial T$.
We have a group homomorphism $\phi \from \pi_1(T)\cong F_2 \to G$, and we can associate to $\ker(\phi)$ a finite, regular, path-connected cover $q \from S \to T$ with base point $s_0 \in q^{-1}(t_0)$ and $q_*(\pi_1(S,s_0))=\ker(\phi)$.
Note that $S$ is again an orientable surface with $\partial S = q^{-1}(\partial T) = q^{-1}(\alpha) = \alpha_1 \cup \ldots \cup \alpha_m $.
The surface $S$ is homotopy-equivalent to $Y$.
For $i=1,\ldots,m$ let $[\alpha_i] \in H_1(S;\C)$ be the homology class represented by $\alpha_i$ and $[\partial S] \coloneqq \sum_{i=1}^m[\alpha_i] \in H_1(S;\C)$ the homology class represented by $\partial S$.
Denote by $\C[\alpha_i]$ and $\C[\partial S]$ their respective $\C$-spans.
Let $\bigoplus_{i=1}^m \C [ \alpha_i ]$ be the $m$-dimensional $G$-representation, where $G$ acts by permuting the homology classes $[\alpha_i]$.
We then have the following short exact sequence of $G$-representations
\[0 \to \C [\partial S] \xrightarrow{\psi} \bigoplus_{i=1}^m \C [ \alpha_i ] \xrightarrow{\varphi} \Span_{H_1(S; \C)}  \{ [\alpha_1], \ldots, [\alpha_m] \} \to 0,\]
where the map $\psi$ is given by $z \mapsto (z,\ldots,z)$, and $\varphi$ sends $(z_1, \ldots, z_m)$ to $\sum_{i=1}^m z_i [\alpha_i]$.
One can verify that $\psi$ and $\varphi$ are $G$-equivariant.
The exactness follows for example from the exactness of the long exact sequence of the pair $(S,\partial S)$.

Recall that we are interested in understanding $\Span_{H_1(S; \C)} \{ [\alpha_1], \ldots, [\alpha_m] \}$.
Using the short exact sequence it is enough to understand $\bigoplus_{i=1}^m \C [ \alpha_i ]$ and $\C [\partial S]$ as $G$-representations.
Let us assume without loss of generality that $\alpha_1$ contains the base point $s_0 \in S$.
Then
\[ \Stab_G(\alpha_1)=\langle \phi(x) \rangle = K \leq G,\]
since $\alpha$ is simple.
The curves $\alpha_1, \ldots, \alpha_m$ are the elevations of $\alpha$, and $G$ permutes these.
Hence we can identify the curves $\alpha_2, \ldots, \alpha_m$ with the left cosets of $K$ in $G$, because $K$ is the stabilizer of $\alpha_1$.
Choose representatives $g_1 = 1_G, g_2, \ldots, g_m \in G$ of the left cosets of $K$ in $G$.
Then
\[\bigoplus_{i=1}^m \C [ \alpha_i ] = \bigoplus_{i=1}^m \C g_i[ \alpha_1 ] \cong \Ind_K^G(\C_\triv)\]
as $G$-representations by the defining property of the induced representation.
Clearly, $G$ acts trivially on $[\partial S]$.
Thus we have proved the proposition, since by the exactness of the short exact sequence above we have
\[\Span_{H_1(S; \C)} \{ [\alpha_1], \ldots, [\alpha_m] \} = \Span_{\C [G]} \{[\alpha_1] \} \cong \Ind_K^G(\C_\triv) / \C_\triv. \qedhere \] 
\end{proof}

It turns out that the above representation-theoretic property is in general not sufficient to characterize homology classes of elevations of primitive commutators.
A counter-example for $n=2$ will be given in Proposition \ref{Propo_InversePropo1.5}, whose construction is based on the idea of stacking covers and the proof of Theorem \ref{Propo2.1_Inverse}.
We need some preliminary considerations.

For a group $G$, a subgroup $H \leq G$ is \emph{characteristic} if it is invariant under every automorphism of $G$, i.e.\ for all $\varphi \in \Aut(G)$ we have $\varphi(H) \subseteq H$.
If we have a sequence of subgroups of the form $K \mathrel{\unlhd}_{\textnormal{char}} H \mathrel{\unlhd} G$ with $K$ characteristic in $H$ and $H$ normal in $G$, then $K$ is also normal in $G$.

\begin{examp} \label{Lemma_modmHomologyCoverIsCharacteristic}
For all $n,m  \geq 2$ the $\modd \, m$-homology cover is a characteristic cover of the wedge of $n$ circles.
Recall that the $\modd \, m$-homology cover is given by the surjective group homomorphism
\[ \phi \from F_n \rightarrow F_n/[F_n,F_n] \cong \Z^n \xrightarrow{\modd \, m} (\Z / m \Z)^n. \]
It suffices to verify that $\ker(\phi)$ is characteristic in $F_n$.
Let $\alpha \in \Aut(F_n)$.
Since $\ker(\phi) = \langle [x,y], z^m \mid x, y, z \in F_n \rangle$, we conclude that  $\alpha( \ker(\phi)) \subseteq \ker(\phi)$, and hence that the cover is characteristic.
\end{examp}

The $\modd \, 2$-homology cover has the additional property that all primitive commutators lift to primitive elements.

\begin{lem} \label{Lemma_mod2HomCover}
Let  $\phi \from F_2 \to (\Z / 2 \Z)^2$ be the $\modd \, 2$-homology cover as in Example \ref{Lemma_modmHomologyCoverIsCharacteristic}.
Then for $x \coloneqq [l,l']$ any primitive commutator in $F_2$ with $l \neq l'$, its preferred elevation is primitive.
\end{lem}
\begin{proof}
We denote by $p \from (Y, y_0) \to (X,x_0)$ the associated finite, characteristic cover, where $(X,x_0)$ is the wedge of two circles and $x_0$ the point of identification.
Let $x_1$ and $x_2$ be generators of the fundamental groups of the two circles.
We will show the claim for $x \coloneqq [x_1, x_2]$ first.
Then $\phi(x)=0$ and thus $x$ lifts to a closed curve on $Y$ illustrated by the dotted lines.
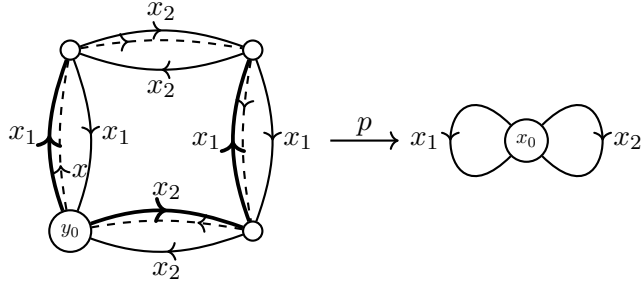
\begin{eqnarray*}
\begin{tikzpicture}[thick,scale=0.60]
\node (node0) at (3,0) [circle,draw=black,scale=0.7] {$x_0$}; 

\node (node1) at (-7,-2) [circle,draw=black,scale=0.7] {$y_0$}; 
\node (node2) at (-7,2) [circle,draw=black,scale=0.7] {}; 
\node (node3) at (-3,2) [circle,draw=black,scale=0.7] {}; 
\node (node4) at (-3,-2) [circle,draw=black,scale=0.7] {}; 

\draw[->-=.5, line width = 1.4pt] (node1) to[out=110,in=-110] node[left] {$x_1$} (node2);
\draw[->-=.5] (node2) to[out=20,in=160] node[above] {$x_2$} (node3);
\draw[->-=.5] (node3) to[out=-70,in=70] node[right] {$x_1$} (node4); 
\draw[->-=.5] (node4) to[out=-160,in=-20] node[below] {$x_2$} (node1);
\draw[->-=.5, line width = 1.4pt] (node1) to[out=20,in=160] node[above] {$x_2$} (node4);
\draw[->-=.5, line width = 1.4pt] (node4) to[out=110,in=-110] node[left] {$x_1$} (node3);
\draw[->-=.5] (node3) to[out=-160,in=-20] node[below] {$x_2$} (node2); 
\draw[->-=.5] (node2) to[out=-70,in=70] node[right] {$x_1$} (node1);

\draw[->] (-1.3,0) -- (0.2,0) node at (-0.55,0.3){$p$} ;

\draw[->-=.5, scale=5] (node0) to[out=45,in=-45,loop] node[auto]{$x_2$} (node0);
\draw[->-=.5, scale=5] (node0) to[out=135,in=-135,loop] node[left]{$x_1$} (node0);

\draw[->-=.3, dashed] (node1) to[out=100,in=-100] node[near start, right] {$x$} (node2);
\draw[->-=.3, dashed] (node2) to[out=10,in=170] (node3);
\draw[->-=.3, dashed] (node3) to[out=-100,in=100] (node4);
\draw[->-=.3, dashed] (node4) to[out=170,in=10] (node1);

\end{tikzpicture}
\end{eqnarray*}
We need to extend $x$ to a free basis of $\pi_1(Y, y_0)$.
For example 
\[\{x, x_1^2, x_2^2, x_2 x_1^2 x_2^{-1}, x_2 x_1 x_2 x_1^{-1} \}\]
is a free basis of $\pi_1(Y, y_0)$.
Namely, consider the spanning tree illustrated by the bold lines in the picture above.

For $x'$ any primitive commutator we use that there exists a homotopy equivalence of $X$ that sends $x'$ to $x$, since the primitive elements defining $x'$ form a basis of $F_2$.
Since the cover is characteristic, this homotopy equivalence can be lifted to $Y$, and the same argument works.
\end{proof}

There is one last remark before we prove Proposition \ref{Propo_InversePropo1.5}.
If $p \from Z \to X$ and $q \from Y \to Z$ are finite, regular covers with base points $z_0 \in p^{-1}(x_0)$ and $y_0 \in q^{-1}(z_0)$, respectively, such that $p \circ q \from Y \to X$ is a finite, regular cover, then for $l$ any loop on $X$, we have by the uniqueness of lifts and the choice of base points that
\[ \tilde{l}_{p \circ q} = \widetilde{(\tilde{l}_{p})}_{q}.\]

\begin{proof}[Proof of Proposition \ref{Propo_InversePropo1.5}]
Recall that for $n=2$ we would like to prove the existence of a finite, regular cover $p \from Y \to X$, together with $z \in H_1(Y;\Z) \cap \ker(p_*)$ indivisible such that
\begin{enumerate}[label=(\roman*)]
\item $\Span_{\C [G]} \{z\} \cong \Ind_{\langle g \rangle}^G(\C_\triv)/\C_\triv $ for some $g \in G$, where $G$ is the deck group associated to the cover $p$, and
\item $z$ cannot be represented by an elevation of a primitive commutator in $X$.
\end{enumerate}
We will construct $Y$ using iterated covers.
Let $\pi_1(X) \cong F_2$ be the free group on the generators $x_1$ and $x_2$.
Consider the primitive commutators $ x' \coloneqq [x_1,x_2], \; x'' \coloneqq [x_2^{-1},x_1]$.
We start with the $\modd \, 2$-homology cover $p_1 \from Z \to X$ defined by the surjective homomorphism
\[F_2 \to Q \coloneqq (\Z / 2 \Z)^2 = \langle A \rangle \times \langle B \rangle, \; x_1 \mapsto A, \, x_2 \mapsto B.\]
Then $l' \coloneqq \tilde{x}'$ and $l'' \coloneqq \tilde{x}''$ are primitive in $\pi_1(Z)$ by Lemma \ref{Lemma_mod2HomCover}.
We additionally have that $l'$ and $l''$ lie in one basis.
For example, the elements $x_1^2$, $x_2^2$ and $x_2 x_1^2 x_2^{-1}$ complete $l'$ and $l''$ to a basis of $\pi_1(Z)\cong F_5$, using the same notation as in the proof of Lemma \ref{Lemma_mod2HomCover}.

Consider now the $\modd \, 2$-homology cover $p_2 \from Y \to Z$ defined by the surjection
\[ F_5 \to N \coloneqq (\Z / 2 \Z)^5 = \langle C \rangle \times \langle D \rangle \times \langle N_1 \rangle \times \langle N_2 \rangle \times \langle N_3 \rangle , \; l' \mapsto C, \, l'' \mapsto D \]
and extension.
Then $p \coloneqq p_1 \circ p_2 \from Y \to X$ is a finite, regular cover because $p_2$ is characteristic, and we denote its deck group by $G$.
We have the following short exact sequence of groups $ 1 \to N \to G \to Q \to 1$.
Since $0=[x']=[x''] \in H_1(X;\Z)$, it follows that
\[ z' \coloneqq \left[ \tilde{x}'_{p} = \tilde{l}'_{p_2} \right], \; z''\coloneqq \left[ \tilde{x}''_{p} = \tilde{l}''_{p_2} \right] \in \ker(p_*).\]
We write $z' = z_1'+\ldots+z'_{2^5}$ and $z''= z''_1+\ldots+z''_{2^5}$ with $ z'_i, z''_i \in M(W_i)$, where $W_1 = \C_\triv, W_2, \ldots, W_{2^5}$ are representatives of the isomorphism classes of the irreducible representations of $N$.
Define $\hat{z} \coloneqq z'-(z'_1-z''_1)$.
Then $\hat{z}$ still lies in the kernel of $p_*$:
Indeed,  since $(p_2)_*(z')=(p_2)_*(z'_1)$ we have
\[(p_2)_*(\hat{z})=(p_2)_*(z'-(z'_1-z''_1)) = (p_2)_*(z''_1)=(p_2)_*(z''),\]
and consequently, $p_*(\hat{z})= (p_1)_* ((p_2)_*(\hat{z}))= (p_1)_* ((p_2)_*(z''))=p_*(z'') = 0$, since $z'' \in \ker(p_*)$.

By the same argument as in the proof of Theorem \ref{Propo2.1_Inverse}, $\hat{z} \in H_1(Y;\Q)$ and we can choose $q\in \Q$ such that $z \coloneqq q \hat{z} \in \ker(p_*) \leq H_1(Y;\Z)$ is indivisible.
Using that $N$ is Abelian, we can now copy the same idea as in the proof of Theorem \ref{Propo2.1_Inverse} to obtain that $z$ is not the homology class of an elevation of a primitive element in $Z$, and thus also not the homology class of an elevation of a primitive commutator in $X$ by Lemma \ref{Lemma_mod2HomCover}.
\end{proof}

\section{Primitive Commutator Homology}
\label{Section_CommutatorHomology}

The first step in attacking the question whether $H_1^\comm(Y;\C) = \ker(p_*)$ is the obstruction in Theorem \ref{Theorem_RepresentationTheoreticObstructionCommutatorHom}.
We adapt the ideas of the proof of \cite[Theorem 1.4]{FarbHensel_FiniteCoversOfGraphs}.

\begin{proof}[Proof of Theorem \ref{Theorem_RepresentationTheoreticObstructionCommutatorHom}]
We would like to prove that if $V$ is an irreducible $G$-representation that appears in $H_1^\comm(Y;\C)$, then $V \in \Irr^\comm(\phi,G) \setminus \{\C_\triv\}$.
Since $H_1^\comm(Y;\C) \leq \ker(p_*) \cong \bigoplus_{i=2}^{k(G)} M(V_i)$ by Lemma \ref{Lemma_TechnicalCoverings1} \ref{item4_technicalCoverings1}, it is clear that $V\neq \C_\triv$.

Let $x=[l,l']$ be a primitive commutator in $F_n$, and set $K \coloneqq \langle \phi(x) \rangle \leq \langle \phi(l), \phi(l') \rangle \eqqcolon H \leq G$.
Let $V$ be an irreducible $G$-representation.
Then writing $V^{K} = \{v \in V : gv=v \textrm{ for all } g \in K \}$ for the space of fixed points of $K$, Proposition \ref{Propo_CommutatorHom} and Frobenius reciprocity (see e.g.\ \cite[Lemma 5.2]{Isaacs_CharacterTheoryOfFiniteGroups})
imply that
\[
\langle \Span_{\C [G]} \{ [\tilde{x}] \}, V \rangle_G \leq \langle \Ind_{K}^G(\C_\triv), V \rangle_G = \langle \C_\triv, \Res_{K}^G(V) \rangle_{K} = \dim(V^{K}),
\]
where the last equality follows from the orthogonality relations.

An irreducible representation $V$ appears in $H_1^\comm(Y;\C)$ if and only if there exists a primitive commutator $x=[l,l']$ in $F_n$ such that 
\[ \langle \Span_{\C [G]} \{ [\tilde{x}] \}, V \rangle_G \neq 0.\]
By the above, this implies $\dim(V^{K}) \neq 0$, which is equivalent to the existence of some non-zero $v \in V$ with $\phi(x)(v) = v$.
Thus, for $V$ to appear in $\Span_{\C [G]} \{ [\tilde{x}] \} $, we necessarily need $V \in \Irr^\comm(\phi,G)$.
\end{proof}

We will continue by studying $H_1^{\comm}(Y;\C)$ in some examples.
If $n=2$ and $G$ is Abelian, we have $\phi(x_{12})=1_G$ and we obtain by the semisimplicity of $M=H_1(Y; \C)$ that
\[ \Span_{\C [G]} \{ [\tilde{x}_{12}] \} \cong \Ind_{\{1_G\}}^G(\C_{\triv}) / \C_{\triv} \cong \C [G] / \C_{\triv} \cong  \bigoplus_{i=2}^{|G|} M(V_i). \]
This implies that $\Span_{\C [G]} \{ [\tilde{x}_{12}] \} \cong \ker(p_*)$ by Lemma \ref{Lemma_TechnicalCoverings1} \ref{item4_technicalCoverings1}, and therefore $H_1^{\comm}(Y;\C)=\ker(p_*)$.

On the other hand, there are examples in rank two with $H_1^{\comm}(Y;\C)\neq \ker(p_*)$, since not every irreducible representation has the property that the image of a primitive commutator has a non-zero fixed point, in other words $\Irr^{\comm}(\phi, G) \neq \Irr(G)$.
\begin{examp} \label{Example}
For $n=2$ any surjective homomorphism $\phi \from F_2 \to S_3$, where $S_3$ denotes the symmetric group on three elements, satisfies $\Irr^{\comm}(\phi, S_3) \neq \Irr(S_3)$.
In fact, the irreducible two-dimensional representation of $S_3$ does not lie in $\Irr^{\comm}(\phi, S_3)$.
It can be realized by $S_3$ acting on $V\coloneqq \C^3/\langle e_1+e_2+e_3\rangle$ by permutation of the coordinates.
Since $\phi$ is surjective the image of a primitive commutator is either the cycle $(1 2 3)$ or $(1 3 2)$.
Computing the action of $(123)$ on $V$, we see that one is not an eigenvalue, thus $(123)$ does not have a fixed vector.
Similarly for $(132)$.
Thus $V \notin \Irr^{\comm}(\phi, S_3)$.
\end{examp}

A more general answer to this question is presented in the following for every $n \geq 2$.
To see that there are examples where $H_1^{\comm}(Y;\C) \neq \ker(p_*)$ for certain covers $p \from Y \to X$, we use iterated covers to relate the question to the one for primitive homology.
Note that in general the composition of two regular covers is not again regular, but every finite, regular cover is itself covered by a finite, characteristic cover.
This is true by the following observation: if $G$ is a finitely generated group and $H \mathrel{\unlhd} G$ a normal subgroup of finite index $k$, then there is a subgroup $H^* \leq H$ which is of finite index in $G$ and characteristic in $G$.
Namely, we can take $H^* \coloneqq \bigcap_{H' \mathrel{\unlhd} G \textnormal{ normal}, \,[G:H']=k} H'$.

In order to prove that in general $H_1^{\comm} \neq \ker(p_*)$, we reduce the case of primitive commutator homology to the case of primitive homology by looking at the $\modd \, 2$-homology cover.
The proof of Lemma \ref{Lemma_mod2HomCover} can be generalised to show that for every $n$ the $\modd \, 2$-homology cover as in Example \ref{Lemma_modmHomologyCoverIsCharacteristic} has the property that primitive commutators lift to primitive elements.
Then we can apply the result of Malestein-Putman in \cite[Example 1.3]{MalesteinPutman_SCCFiniteCoversOfSurfaces} to conclude.
For this we need to understand how primitive homology behaves in iterated covers.
To avoid confusion we introduce the following notation:
If $Y \to X$ is a finite, regular cover of finite graphs, we will from now on
write $H_1^S(Y \to X;\C)$ instead of $H_1^S(Y;\C)$.

\begin{lem} \label{Lemma_PrimToPrim}
Let $q \from Z \to Y$ and $p \from Y \to X$ be finite, regular covers of finite graphs such that $p \circ q \from Z \to X$ is a finite, regular cover.
Then
\[ q_*(H_1^\prim(Z \to X;\C)) \subseteq H_1^\prim(Y \to X; \C). \]
In other words, primitive homology gets mapped into primitive homology.
\end{lem}
\begin{proof}
For $l$ a primitive element in $\pi_1(X)$, we have for some $k \in \Z$
\[ q_* \left( \left[  \tilde{l}_{p\circ q} \right] \right) = k  \left[ \tilde{l}_{p} \right] \in H_1^\prim(Y \to X;\C). \qedhere \]
\end{proof}

\begin{corol}
Let the setup be as in Lemma \ref{Lemma_PrimToPrim} with the condition that $H_1^{\prim}(Y \to X;\C) \neq H_1(Y;\C)$ as $K$-representations where $K \coloneqq \Deck(Y,p)$.
Then 
\[ H_1^{\prim}(Z \to X; \C) \neq H_1(Z;\C) \]
as $G$-representations with $G \coloneqq \Deck(Z,p \circ q)$.
\end{corol}
\begin{proof}
This is straightforward using that $q_* \from H_1(Z;\C) \to H_1(Y;\C)$ is surjective.
\end{proof}

The last corollary implies that we can assume without loss of generality that the cover constructed by Malestein and Putman in \cite[Example 1.3]{MalesteinPutman_SCCFiniteCoversOfSurfaces} is characteristic by passing to a finite, characteristic cover, which will satisfy that primitive homology is not all of homology.
We bring everything together in the following proof.

\begin{proof}[Proof of Theorem \ref{Thm_H1CommNeqKer}]
Recall that for every $n \geq 2$ we would like to construct a finite, regular cover $p \from Z \to X$ such that $H_1^{\comm}(Z \to X;\C) \neq \ker(p_*)$.

Let $p' \from Y \to X$ be the $\modd \, 2$-homology cover with group of deck transformations $K= (\Z / 2 \Z)^n$.
By Example \ref{Lemma_modmHomologyCoverIsCharacteristic}, we know that this cover is characteristic.
Lemma \ref{Lemma_mod2HomCover} tells us that all primitive commutators in $X$ lift to primitive elements in $Y$.
The space $Y$ is again a connected graph with free fundamental group, so we can apply the result of Malestein and Putman, see \cite[Theorem C, Example 1.3]{MalesteinPutman_SCCFiniteCoversOfSurfaces}, to find a finite, regular cover $q \from Z \to Y$ with deck group $H$ such that
\[ H_1^{\prim}(Z \to Y; \C) \neq H_1(Z;\C)\]
as $H$-representations.
By the above considerations we can assume that this cover is characteristic.
Therefore,  $p \coloneqq p' \circ q \from Z \to X$ is regular.
Denote its deck group by $G$.
\begin{equation*}
\begin{tikzcd}
Z \arrow{d}{q \textnormal{ (finite, characteristic cover of Malestein-Putman)}} \arrow[bend right=60,swap]{dd}{p \textnormal{ (finite, regular)}} \\
Y \arrow{d}{p' \; (\modd \, 2 \textnormal{-homology cover})} \\
X
\end{tikzcd}
\end{equation*}
Let us assume for a contradiction that $H_1^{\comm}(Z \to X;\C) = \ker(p_*)$.
Then Lemma \ref{Lemma_TechnicalCoverings1} \ref{item4_technicalCoverings1} implies that $H_1^{\comm}(Z \to X;\C) = \bigoplus_{i=2}^{k(G)} M(V_i),$
with $M=H_1(Z;\C)$ as $G$-representations.
We also know by Lemma \ref{Lemma_TechnicalCoverings1} \ref{item2_technicalCoverings1} that 
\[ M(V_1) = p_\#(H_1(X;\C)) = p_\# \left( \Span_\C \{ [l] \mid  l \in \pi_1(X) \textnormal{ primitive} \} \right). \]
By definition, we have $p_\#([l]) = \sum_{g \in G} g \,  [\tilde{l}]$.
For $l$ a primitive element, it follows that $p_\#([l]) \in H_1^{\prim}(Z;\C)$.
Thus we obtain $M(V_1) \leq H_1^{\prim}(Z \to X; \C)$.
This implies that
\[ H_1(Z;\C) = H_1^{\prim}(Z \to X;\C) + H_1^{\comm}(Z \to X;\C).\]
We would like to conclude from this equality that $H_1(Z;\C) \leq H_1^{\prim}(Z \to Y;\C)$, which contradicts the choice of the cover $q \from Z \to Y$.
We know that lifts of primitive elements are primitive, and by Lemma \ref{Lemma_mod2HomCover} that for the $\modd \, 2$-homology cover primitive commutators in $X$ also lift to primitive elements in $Y$.
Set 
\[ S \coloneqq \left\lbrace \tilde{x}_{p'} \mid x \in \pi_1(X) \textnormal{ primitive or a primitive commutator} \right\rbrace \subseteq \pi_1(Y).\]
Then $S \subseteq \{l \in \pi_1(Y) \textnormal{ primitive} \}$.
Clearly, for any two subsets $S' \subseteq S'' \subseteq F_m$ we have $H_1^{S'}(Y;\C) \leq H_1^{S''}(Y;\C)$.
By the above equation together with the last observation, we obtain
\[ H_1(Z;\C) = H_1^S(Z \to Y; \C) \leq H_1^{\prim}(Z \to Y;\C). \qedhere \]
\end{proof}


\bibliographystyle{amsalpha}
\bibliography{bibl}{}

\end{document}